\documentclass[10pt]{article}

\usepackage{amssymb}
\usepackage{amsmath}
\usepackage{amsthm}

\theoremstyle{plain}
   \newtheorem {thm}{Theorem}[section]
   
   \newtheorem*{RR}{The Rogers-Ramanujan Identities}
   \newtheorem*{aRR}{The $a$-generalized Rogers-Ramanujan Identities}
   \newtheorem*{BaileyL}{Bailey's Lemma}

   \newtheorem{lemma}[thm]{Lemma}
   \newtheorem{prop}[thm]{Proposition}

\theoremstyle{definition}
   
   \newtheorem {rmk}[thm]{Remark}
   \newtheorem{definition}[thm]{Definition}

\numberwithin{equation}{section}

\newcommand{\binomial}[2]{ \genfrac{(}{)}{0pt}{}{#1}{#2} }

\title{$q$-Difference Equations and \\
Identities of the Rogers-Ramanujan-Bailey Type}
\author{Andrew V. Sills\\  
\small{Department of Mathematics, Rutgers University}\\
\small{110 Frelinghuysen Road}\\
\small{Hill Center--Busch Campus, Piscataway, NJ 08854}\\ \\
\small{\textit{Phone:} 1-732-445-3488, \textit{Fax:} 1-732-445-5530}\\ \\
\small{\textit{e-mail:} \texttt{asills@math.rutgers.edu}}\\
\small{\textit{URL:} \texttt{http://www.math.rutgers.edu/\~{}asills}}\\  \\
\small{\textit{2000 AMS Subject Codes:} 11B65. 39A13, 05A19, 33D15 }}
\date{June 21, 2004}

\allowdisplaybreaks[2]

\begin{document}
\maketitle

\begin{abstract}
 In a recent paper, I defined the ``standard multiparameter Bailey pair" 
(SMPBP) and
demonstrated that all of the classical Bailey pairs considered by W.N. Bailey
in his famous paper (\textit{Proc. London Math. Soc. (2)}, 
\textbf{50} (1948), 1--10)
arose as special cases of the SMPBP.  
Additionally, I was able to find a number of
new Rogers-Ramanujan type identities.   From a given Bailey pair, normally
only one or two Rogers-Ramanujan type identities follow immediately.  In this
present work, I present the set of $q$-difference equations associated with 
the SMPBP, and use these $q$-difference equations to deduce the 
complete families of 
Rogers-Ramanujan type identities.
\end{abstract}

\section{Introduction}
\subsection{Overview} \label{ov}
Recall the famous Rogers-Ramanujan identities:
\begin{RR} 
 \begin{equation}\label{RR1}
   \sum_{n=0}^\infty \frac{q^{n^2}}{(q;q)_n} = 
   \frac{(q^2, q^3, q^5; q^5)_\infty}{(q;q)_\infty},
 \end{equation} and
\begin{equation}\label{RR2}
   \sum_{n=0}^\infty \frac{q^{n^2+n}}{(q;q)_n} = 
   \frac{(q, q^4, q^5; q^5)_\infty}{(q;q)_\infty},
 \end{equation}
where \[ (a;q)_m = \prod_{j=0}^{m-1} (1-aq^j), \]
      \[ (a;q)_\infty = \prod_{j=0}^\infty (1-aq^j), \] and
      \[ (a_1, a_2, \dots, a_r; q)_\infty 
= (a_1;q)_\infty (a_2;q)_\infty \dots (a_r;q)_\infty, \]
\end{RR} 
(Although the results in this paper may be considered purely from the
point of view of formal power series, they also yield identities of
analytic functions provided $|q|<1$.)

The Rogers-Ramanujan identities were discovered by L.~J.~Rogers~\cite{ljr1894},
and were rediscovered independently by S. Ramanujan~\cite{MacMahon} and
I. Schur~\cite{Schur}.  There are many series--product 
identities similar in form to the Rogers-Ramanujan identities, and 
were dubbed ``identities of the Rogers-Ramanujan type" by W. N. Bailey in~\cite{wnb2}.  

   Furthermore, Bailey found what he termed ``$a$-generalizations" of a number of
Rogers-Ramanujan type identities.  For example, the $a$-generalizations of
(\ref{RR1}) and (\ref{RR2}) are, respectively,
\begin{aRR} 
 \begin{equation}\label{aRR1}
   \sum_{n=0}^\infty \frac{a^{n} q^{n^2}}{(q;q)_n} = 
    \frac{1}{(aq;q)_\infty} \sum_{n=0}^\infty 
     \frac{(-1)^n a^{2n} q^{n(5n-1)/2} (1-aq^{2n}) (a;q)_n}{(1-a)(q;q)_n},
 \end{equation} and
\begin{equation}\label{aRR2}
   \sum_{n=0}^\infty \frac{ a^n q^{n(n+1)}}{(q;q)_n} = 
   \frac{1}{(aq;q)_\infty} \sum_{n=0}^\infty
 \frac{(-1)^n a^{2n} q^{n(5n+3)/2} (1-aq^{2n+1}) (aq;q)_n}{(q;q)_n},
 \end{equation}
\end{aRR} 

Notice that (\ref{RR1}) and (\ref{RR2}) are obtained from (\ref{aRR1}) and
(\ref{aRR2}) respectively by setting $a=1$ and applying Jacobi's triple product identity 
\cite[p. 12, equation (1.6.1)]{gr} to the right hand side.

If we define 
  \[ F_1(a) := F_1(a,q):= \sum_{n\geqq 0} \frac{a^{n} q^{n(n+1)}}{(q;q)_n} \]
 and 
   \[ F_2(a) := F_2(a,q):= \sum_{n\geqq 0} \frac{a^{n} q^{n^2}}{(q;q)_n}, \]
 it is well known (and easy to see) that $F_1(a)$ and $F_2(a)$ satisfy the following system of
 $q$-difference equations:
 
 \begin{eqnarray}
  F_1(a) &=& F_2(aq) \label{aRRqDiff1}\\  
  F_2(a) &=& F_1(a) + aq F_2(aq) \label{aRRqDiff2}
 \end{eqnarray}
 
 with initial conditions $F_1(0) = F_2(0) = 1$.
 
 A standard proof of $(\ref{aRR1})$ and $(\ref{aRR2})$ is to show that 
the right hand sides of $(\ref{aRR1})$ and $(\ref{aRR2})$ satisfy the same
$q$-difference equations and initial conditions as $F_1(a)$ and $F_2(a)$;
see, e.g., ~\cite[p. 183 ff.]{gea:numth}.
 
  The main goal of this paper is to show that the system of $q$-difference 
equations
given above is a special case of a much more general form, and that once this
general form is established, we may use it to derive new Rogers-Ramanujan-Bailey
type identities related to those studied in~\cite{avs:rrb}.

 More specifically, we make the following definition:
 \begin{definition} For $d$, $e$, $k \in \mathbb{Z}_+$, and
$1\leqq i \leqq k+d(e-1)$, let
\begin{equation} \label{Qdef} 
 Q^{(d,e,k)}_{i} (a) := \frac{1}{(a^e q^e; q^e)_\infty}
   \sum_{n\geqq 0} \frac{(-1)^n a^{[k+d(e-1)]n}  
    q^{[k+d(e-1)+\frac 12]dn^2 + [k+d(e-1)+\frac 12 -i ]dn } (1-a^i q^{(2n+1)di}) 
    (aq^d;q^d)_n}{(q^d;q^d)_n},
\end{equation}
\end{definition}
and subsequently we shall prove that the following theorem:

\begin{thm}\label{qde}
The $Q^{(d,e,k)}_i (a)$ satisfy the following system of $q$-difference equations
\begin{eqnarray}
  Q^{(d,e,k)}_{1} (a) &=& \frac{1-aq^d}{(a^e q^e;q^e)_d} 
Q^{(d,e,k)}_{de-d+k} (aq^d), \mbox{\ and}
   \label{qde1}\\
  Q^{(d,e,k)}_{i} (a) &=& Q^{(d,e,k)}_{i-1} (a)
    + \frac{(1-aq^d) a^{i-1} q^{d(i-1)}}{(a^e q^e;q^e)_d} 
Q^{(d,e,k)}_{de-d+k-i+1} 
    (aq^d) \label{qde2}
\end{eqnarray}
for  \[ 2\leqq i\leqq de-d+k,\]
with initial conditions 
\begin{equation} \label{initconds}
 Q^{(d,e,k)}_1 (0) = Q^{(d,e,k)}_2 (0) = \cdots = Q^{(d,e,k)}_{de-d+k} (0) = 1.
  \end{equation}
  \end{thm}
  
Thus (\ref{qde1}), (\ref{qde2}), and (\ref{initconds}) uniquely determine 
$Q^{(d,e,k)}_i (a)$ as a double power series in $a$ and $q$.
  
We then prove that for particular values of $d$, $e$, and $k$,  various seemingly
different functions $F^{(d,e,k)}_i (a)$, to be defined later, satisfy the same 
$q$-difference equations and initial conditions.  From this,
families of Rogers-Ramanujan type identities are established.

 In order to motivate the definition of the various $F^{(d,e,k)}_i (a)$, we will
need to review some background material in \S\ref{background}.  Next, in 
\S\ref{Qpf}, Theorem~\ref{qde} will be proved.  In \S\ref{F}, several instances
of the $F^{(d,e,k)}_i (a)$ will be derived, from which families of 
Rogers-Ramanujan
type identities will be deduced.
In \S\ref{conc}, I comment on how this work fits into
context with previous work, and suggest a promising direction for further
research.
  
\subsection{Background}\label{background}

\begin{definition}
A pair of sequences $\left(\alpha_n (a,q),\beta_n(a,q)\right)$ is called a 
\emph{Bailey pair} if
for $n\geqq 0$, 
   \begin{equation} \label{BPdef}
      \beta_n (a,q) = \sum_{r=0}^n \frac{\alpha_r(a,q) }{(q;q)_{n-r} (aq;q)_{n+r}}.
   \end{equation}
\end{definition}
  
In~\cite{wnb1} and~\cite{wnb2}, Bailey proved the fundamental
result now known as ``Bailey's lemma" (see also~\cite[Chapter 3]{GEA:qs}).
\begin{BaileyL}
If $(\alpha_r (a,q), \beta_j (a,q))$ form a Bailey pair, then
\begin{gather} 
  \frac{1}{ (\frac{aq}{\rho_1};q)_n ( \frac{aq}{\rho_2};q)_n} 
  \sum_{j\geqq 0} \frac{ (\rho_1;q)_j (\rho_2;q)_j 
     (\frac{aq}{\rho_1 \rho_2} ;q)_{n-j}}
     {(q;q)_{n-j}} \left( \frac{aq}{\rho_1 \rho_2} \right)^j \beta_j(a,q) \nonumber\\
 = \sum_{r=0}^n \frac{ (\rho_1;q)_r (\rho_2;q)_r}
    { (\frac{aq}{\rho_1};q)_r (\frac{aq}{\rho_2};q)_r (q;q)_{n-r} (aq;q)_{n+r}}
    \left( \frac{aq}{\rho_1 \rho_2} \right)^r \alpha_r(a,q). \label{BL}
\end{gather}
\end{BaileyL}

In~\cite{avs:rrb}, I defined the ``standard multiparameter Bailey pair" (SMPBP) via
\begin{equation}  \label{alphadef}
    \alpha^{(d,e,k)}_{n} (a,b,q) := 
        \left\{  \begin{array}{ll}
     \parbox{4cm}{\[ \frac{a^{(k-d+1)r/e} q^{(dk - d^2 + d) r^2/e} 
       (a^{1/e} q^{2d/e};q^{2d/e})_r (a^{1/e};q^{d/e})_r}
       {b^{r/e} {(a^{1/e} b^{-1/e} q^{d/e}; q^{d/e})}
        (a^{1/e} ;q^{2d/e})_r (q^{d/e};q^{d/e})_r }, \]}\\
            &\mbox{if $n= dr$, and} \\
      0,                                      &\mbox{otherwise,}
              \end{array} \right.
    \end{equation}
 with the corresponding $\beta^{(d,e,k)}_n (a,b,q)$ determined by (\ref{BPdef}),
 i.e.
 \begin{equation}\label{SMPBPbetadef}
   \beta_n^{(d,e,k)} (a,b,q) 
   := \sum_{r=0}^n \frac{\alpha_r^{(d,e,k)}(a,b,q)}{(q;q)_{n-r} (aq;q)_{n+r}},
   \end{equation}
 and showed that by specializing $d$, $e$, and $k$ to particular values, we could
 recover all of the Bailey pairs studied by Bailey in ~\cite{wnb1} and ~\cite{wnb2},
 and additionally derive many new Bailey pairs which lead to elegant new 
 Rogers-Ramanujan type identities.  As usual, Rogers-Ramanujan type identities
 result from inserting Bailey pairs into certain limiting cases of Bailey's lemma.
 Here, we will focus on the limiting case $n,\rho_1,\rho_2\to\infty$ of (\ref{BL}),
 with the $b=0$ case of the SMPBP as the Bailey pair inserted, i.e.
 \begin{equation} \label{WBL}
   \sum_{n\geqq 0} a^{en} q^{en^2} \beta^{(d,e,k)}_n(a^e,0,q^e) =
   \frac{1}{(a^e q^e;q^e)_\infty} \sum_{n\geqq 0} a^{en} q^{en^2}
   \alpha^{(d,e,k)}_n (a^e,0,q^e).  
 \end{equation} 
 
 \begin{rmk}
 It will become clear subsequently that  $Q^{(d,e,k)}_{de-d+k} (a)$ 
 in fact equals the right hand side 
 of (\ref{WBL}).
 \end{rmk}
 
 \begin{rmk}
 It is well known that Rogers-Ramanujan type identities 
 end to occur in closely related families,
 e.g. there are two Rogers-Ramanujan identities~(\ref{RR1}), (\ref{RR2}), 
 three Rogers-Selberg identities
 ~\cite[p. 421, equations (1.3)--(1.5)]{wnb1},
 four Dyson mod 27 identities~\cite[p. 433, equations (B1--B4)]{wnb1}, etc.  
 However, only one or two members of a family
 can normally be determined immediately from an instance of the SMPBP.  We 
 can derive all $k+d(e-1)$ members of a given family with the aid of 
 (\ref{qde1}) and (\ref{qde2}).
 \end{rmk}
 
 \section{Proof of Theorem~\ref{qde} and related results}\label{Qpf}
 Before launching into the proof of Theorem~\ref{qde}, let us establish
 the following lemma:
 \begin{lemma} 
 For $d$, $e$, $k\in \mathbb Z_+$, and $1\leqq i \leqq k+d(e-1)$,
 \begin{equation}\label{lem}
 Q^{(d,e,k)}_{de-d+k} (a) = \frac{1}{(a^e q^e; q^e)_\infty}
 \sum_{n\geqq 0} \frac{(-1)^n a^{(de-d+k)n} q^{(de-d+k+\frac 12)dn^2 - \frac d2 n}
 (1-aq^{2dn}) (a;q^d)_n}{(1-a)(q^d;q^d)_n}.
 \end{equation}
 \end{lemma}
 
 \begin{proof}
 \begin{eqnarray*}
 &&\sum_{n\geqq 0} \frac{(-1)^n a^{(de-d+k)n} 
 q^{(de-d+k+\frac 12)dn^2 - \frac d2 n}
 (1-aq^{2dn}) (a;q^d)_n}{(1-a)(q^d;q^d)_n}\\
 &=& \sum_{n\geqq 0} \frac{(-1)^n a^{(de-d+k)n} 
 q^{(de-d+k+\frac 12)dn^2 - \frac d2 n}
 (a;q^d)_n}{(1-a)(q^d;q^d)_n} \left\{ q^{dn}(1-aq^{dn}) + (1-q^{dn}) \right\} \\
 &=& \sum_{n\geqq 0} \frac{(-1)^n a^{(de-d+k)n} 
 q^{(de-d+k+\frac 12)dn^2 + \frac d2 n}
 (aq^d;q^d)_n}{(q^d;q^d)_n} \\
&& \qquad\qquad+ \sum_{n\geqq 1} \frac{(-1)^n a^{(de-d+k)n} 
 q^{(de-d+k+\frac 12)dn^2 - \frac d2 n}
 (aq^d;q^d)_{n-1}}{(q^d;q^d)_{n-1}} \\
 &=& \sum_{n\geqq 0} \frac{(-1)^n a^{(de-d+k)n} 
 q^{(de-d+k+\frac 12)dn^2 + \frac d2 n}
 (aq^d;q^d)_n}{(q^d;q^d)_n}\\ 
&&\qquad\qquad- \sum_{n\geqq 0} \frac{(-1)^n a^{(de-d+k)n + (de-d+k)} 
 q^{(de-d+k+\frac 12)dn^2 + (2k-2d-2de+\frac 12)dn + (de-d+k)d}
 (aq^d;q^d)_n}{(q^d;q^d)_n} \\
 &=& \sum_{n\geqq 0} \frac{(-1)^n a^{(de-d+k)n} 
 q^{(de-d+k+\frac 12)dn^2 + \frac d2 n}
 (aq^d;q^d)_n  (1-a^{de-d+k} q^{(de-d+k)(2n+1)d})}{(q^d;q^d)_n}\\ 
 &=& (a^e q^e; q^e)_\infty Q^{(d,e,k)}_{de-d+k} (a)
 \end{eqnarray*}
 \end{proof}
 
 \begin{proof}[Proof of Theorem~\ref{qde}] 
 First, (\ref{qde1}) is easily established:
  \begin{proof}[Proof of (\ref{qde1})]
 \begin{eqnarray*}
 &&\frac{1-aq^d}{(a^e q^e;q^e)_d} Q^{(d,e,k)}_{de-d+k} (a q^d)\\ 
 &=& \frac{1-aq^d}{(a^e q^e; q^e)_d} 
 \frac{1}{(a^e q^{(d+1)e};q^e)_\infty}
 \sum_{n\geqq 0} \frac{(-1)^n a^{(de-d+k)n} 
 q^{(de-d+k+\frac 12)dn^2 + (de-d+k-\frac 12) dn}
 (aq^d;q^d)_n  (1-aq^{(2n+1)d})}{(1-aq^d) (q^d;q^d)_n}\\
 &&\qquad\qquad\qquad\qquad\mbox{(by Lemma~\ref{lem})}\\
 &=& \frac{1}{(a^e q^e; q^e)_\infty} 
 \sum_{n\geqq 0} \frac{(-1)^n a^{(de-d+k)n} 
 q^{(de-d+k+\frac 12)dn^2 + (de-d+k-\frac 12) dn}
 (aq^d;q^d)_n  (1-aq^{(2n+1)d})}{(q^d;q^d)_n}\\
 &=& Q^{(d,e,k)}_1 (a).
 \end{eqnarray*}
 \end{proof}
Next, we establish (\ref{qde2}):
\begin{proof}[Proof of (\ref{qde2})]
\begin{eqnarray*}
 && (a^e q^e; q^e)_\infty \Big( Q^{(d,e,k)}_i (a) - Q^{(d,e,k)}_{i-1}(a) \Big) \\
 &=& 
 \sum_{n\geqq 0} \frac{(-1)^n a^{(de-d+k)n} 
    q^{(de-d+k+\frac 12)dn^2 + (de-d+k+\frac 12 -i )dn } (1-a^i q^{(2n+1)di}) 
    (aq^d;q^d)_n}{(q^d;q^d)_n}\\
&&\qquad - \sum_{n\geqq 0} \frac{(-1)^n a^{(de-d+k)n} 
    q^{(de-d+k+\frac 12)dn^2 + (de-d+k+\frac 32 -i )dn } (1-a^{i-1} q^{(2n+1)d(i-1)}) 
    (aq^d;q^d)_n}{(q^d;q^d)_n}\\
&=& \sum_{n\geqq 0} \frac{(-1)^n a^{(de-d+k)n} 
    q^{(de-d+k+\frac 12)dn^2 + (de-d+k+\frac 12)dn }  
    (aq^d;q^d)_n}{(q^d;q^d)_n}\\
&&\qquad \times \Big\{ q^{-idn}(1-a^i q^{(2n+1)di}) - q^{(1-i)dn} 
   (1-a^{i-1} q^{(2n+1)d(i-1)} ) \Big\}\\
&=& \sum_{n\geqq 0} \frac{(-1)^n a^{(de-d+k)n} 
    q^{(de-d+k+\frac 12)dn^2 + (de-d+k+\frac 12)dn }  
    (aq^d;q^d)_n}{(q^d;q^d)_n}\\
&&\qquad \times \Big\{ q^{-idn}(1-q^{dn}) +  a^{i-1} q^{dni-nd+di-d}
(1-a q^{d(n+1)}) \Big\}\\
&=& \sum_{n\geqq 1} \frac{(-1)^n a^{(de-d+k)n} 
    q^{(de-d+k+\frac 12)dn^2 + (de-d+k+\frac 12-i)dn }  
    (aq^d;q^d)_n}{(q^d;q^d)_{n-1}}\\
&& \qquad +
\sum_{n\geqq 0} \frac{(-1)^n a^{(de-d+k)n+i-1} 
    q^{(de-d+k+\frac 12)dn^2 + (de-d+k-\frac 12+i)dn + d(i-1) }  
    (aq^d;q^d)_{n+1}}{(q^d;q^d)_{n}}\\
&=& -\sum_{n\geqq 0} \frac{(-1)^n a^{(de-d+k)n +(de-d+k)} 
    q^{(de-d+k+\frac 12)dn^2 + (3de-3d+3k+\frac 32-i)dn + d(2de-2d+2k-i+1) }  
    (aq^d;q^d)_{n+1}}{(q^d;q^d)_{n}}\\
&& \qquad +
\sum_{n\geqq 0} \frac{(-1)^n a^{(de-d+k)n+i-1} 
    q^{(de-d+k+\frac 12)dn^2 + (de-d+k-\frac 12+i)dn + d(i-1) }  
    (aq^d;q^d)_{n+1}}{(q^d;q^d)_{n}}\\
&=& (1-aq^d)a^{i-1}q^{d(i-1)}\left(
    \sum_{n\geqq 0} \frac{(-1)^n a^{(de-d+k)n} 
    q^{(de-d+k+\frac 12)dn^2 + (de-d+k-\frac 12+i)dn}  
    (aq^{2d};q^d)_{n}}{(q^d;q^d)_{n}} \right.\\
  &&\quad \left. -\sum_{n\geqq 0} \frac{(-1)^n a^{(de-d+k)n +(de-d+k-i+1)} 
    q^{(de-d+k+\frac 12)dn^2 + (3de-3d+3k+\frac 32-i)dn + 2d(de-d+k-i+1) }  
    (aq^{2d};q^d)_n}{(q^d;q^d)_{n}} \right)\\
&=& (1-aq^d)a^{i-1}q^{d(i-1)}
    \sum_{n\geqq 0} \frac{(-1)^n a^{(de-d+k)n} 
    q^{(de-d+k+\frac 12)dn^2+ (de-d+k-\frac 12+i)dn  }  
    (aq^{2d};q^d)_{n}}{(q^d;q^d)_{n}}\\   
    &&\qquad\qquad\qquad\times(1-a^{de-d+k-i+1} q^{2d(de-d+k+1-i)(n+1)}) \\
 &=& (a^e q^e;q^e)_\infty (1-aq^d)a^{i-1}q^{d(i-1)}  Q^{(d,e,k)}_{de-d+k-i+1} (a q^d).
 \end{eqnarray*}
\end{proof} 
 The above proofs of (\ref{qde1}) and (\ref{qde2}) together with the routine 
 verification that (\ref{initconds}) holds by (\ref{Qdef}) establishes 
 Theorem~\ref{qde}.
 \end{proof}
 
 The right hand sides of Rogers-Ramanujan type identities (in $q$ only) are 
 expressible as infinite products.  Accordingly, we will need the following
 proposition.
 \begin{prop} \label{Q1}
 \[
 Q^{(d,e,k)}_i (1) = 
 \frac{( q^{di},q^{d(2ed-2d+2k+1-i)}, q^{d(2ed-2d+2k+1)} ;
 q^{d(2ed-2d+2k+1)})_\infty}{(q^e;q^e)_\infty}. \]
 \end{prop}
 \begin{proof}
 \begin{eqnarray*}
 & & (q^e;q^e)_\infty Q^{(d,e,k)}(1)\\
 &=& \sum_{n\geqq 0} (-1)^n   
    q^{(k+de-d+\frac 12)dn^2 + (k+de-d+\frac 12 -i )dn } (1-q^{(2n+1)di}) \\
 &=& \sum_{n\geqq 0} (-1)^n   
    q^{(k+de-d+\frac 12)dn^2 + (k+de-d+\frac 12 -i )dn }  
   -\sum_{n\geqq 0} (-1)^n   
    q^{(k+de-d+\frac 12)dn^2 + (k+de-d+\frac 12 -i )dn+(2n+1)di }  \\ 
  &=& \sum_{n\geqq 0} (-1)^n   
    q^{(k+de-d+\frac 12)dn^2 - (-k-de+d-\frac 12 +i )dn }  
   +\sum_{n\geqq 1} (-1)^n   
    q^{(k+de-d+\frac 12)dn^2 + (-k-de+d-\frac 12 +i )dn}\\
  &=& \sum_{n=-\infty}^\infty (-1)^n   
    q^{(k+de-d+\frac 12)dn^2 - (-k-de+d-\frac 12 +i )dn }  \\
  &=& ( q^{di},q^{d(2ed-2d+2k+1-i)}, q^{d(2ed-2d+2k+1)} ;
 q^{d(2ed-2d+2k+1)})_\infty\\
  &&\qquad\qquad\qquad
   \mbox{(by Jacobi's triple product identity~\cite[p. 12, (1.6.1)]{gr}. }
  \end{eqnarray*}

 \end{proof}
 
 \section{Rogers-Ramanujan-Bailey type identities}\label{F}
 \subsection{The general procedure}
First, define $F^{(d,e,k)}_{k+de-d}(a)$ to be the left hand side of (\ref{SMPBPbetadef}):
\begin{equation}\label{Flastgen}
  F^{(d,e,k)}_{k+de-d}(a) := F^{(d,e,k)}_{k+de-d}(a,q) := 
  \sum_{n\geqq 0} a^{en} q^{en} \beta^{(d,e,k)}_n (a^e, 0, q^e).
\end{equation}
Then allow $F^{(d,e,k)}_1 (a)$, $F^{(d,e,k)}_2 (a)$, \dots,  $F^{(d,e,k)}_{k+de-d-1} (a)$
to be determined by 
\begin{eqnarray}
  F^{(d,e,k)}_{1} (a) &=& \frac{1-aq^d}{(a^e q^e;q^e)_d} F^{(d,e,k)}_{de-d+k} (aq^d)
   \label{Fqde1}\\
  F^{(d,e,k)}_{i} (a) &=& F^{(d,e,k)}_{i-1} (a)
    + \frac{(1-aq^d) a^{i-1} q^{d(i-1)}}{(a^e q^e;q^e)_d} F^{(d,e,k)}_{de-d+k-i+1} 
    (aq^d) \label{Fqde2}
\end{eqnarray}
for  \[ 2\leqq i\leqq de-d+k.\]
This will establish the identities 
  \[ F^{(d,e,k)}_i (a) = Q^{(d,e,k)}_i (a), \]
provided the initial conditions 
\begin{equation} \label{Finitconds}
 F^{(d,e,k)}_1 (0) = F^{(d,e,k)}_2 (0) = \cdots = F^{(d,e,k)}_{de-d+k} (0) = 1
  \end{equation}
 are satisfied.
 
 \subsection{Examples}
  In \cite{avs:rrb}, I derived eighteen new Bailey pairs by explicitly finding
 (\ref{SMPBPbetadef}) for various values of $d$, $e$, and $k$, and deduced
 one or two identities associated for each.  With the $q$-difference equations
 now in hand, the full set of $k+d(e-1)$ identities can be deduced for a
 given $(d,e,k)$.

 \subsubsection{Special cases previously in the literature}
 Before displaying the new families of results associated with the Bailey pairs
 found in \cite{avs:rrb}, I should point out that many known families of 
 results, including classical results due to L. J. Rogers, arise from special cases of the
 standard multiparameter Bailey pair (hence its designation as ``standard").
 I summarize these in the following table:

  \begin{tabular}{|l|c|c|c|}
    \hline
  $(d,e,k)$ & Author/Identity & modulus & Reference \\
    \hline
  $(1,1,2)$ &Rogers-Ramanujan & 5 & \cite[p. 331--332]{ljr1894}\\
  $(1,1,k)$ &G.E. Andrews    &$2k+1$& \cite[p. 4082, (1.7)]{GEA:oddmoduli}\\  
  $(1,2,1)$ &L.J. Rogers      & 5 & \cite[pp. 331, 339]{ljr1894}\\
  $(1,2,2)$ &Rogers-Selberg   & 7 & \cite[p. 339, 342]{ljr1894}; \cite[p. 331]{ljr1917}\\
  $(1,3,2)$ &W.N. Bailey      & 9 & \cite[p. 422, (1.6)--(1.8)]{wnb1}\\
  $(1,6,3)$ &Verma-Jain       &17  &\cite[p. 247-248, (3.1)--(3.8)]{vj1}\\
  $(1,6,4)$ &Verma-Jain       & 19   &\cite[p. 248-250, (3.9)--(3.17]{vj1}\\
  $(2,1,1)$ &A.V. Sills        & 6 & \cite[(A.1)]{avs:rrt}\\  
  $(2,1,2)$ &L.J. Rogers      & 10 & \cite[p. 330 (2) lines 2, 3]{ljr1917}\\
  $(2,1,3)$ &L.J. Rogers      & 14 & \cite[p. 341, ex. 2]{ljr1894};
\cite[p. 329 (1) lines 2--3]{ljr1917}\\
  $(2,1,4)$ &A.V. Sills        & 18 & \cite[(A.4)--(A.7)]{avs:rrt}\\
  $(2,1,5)$ &Verma-Jain        & 22 & \cite[pp. 250-251, (3.18)--(3.22)]{vj1}\\
  $(2,3,4)$ &Verma-Jain     & 34 & \cite[pp. 252-253, (3.30)--(3.37)]{vj1}\\
  $(2,3,5)$ &Verma-Jain     & 38 & \cite[pp. 253-255, (3.39)--(3.46)]{vj1}\\
  $(3,1,3)$ &A.V. Sills       & 21 & \cite[(A.8)--(A.10)]{avs:rrt}\\ 
  $(3,1,4)$ &F.J. Dyson       & 27 & \cite[p. 434, (B1)--(B4)]{wnb1}\\
  $(3,1,5)$ &Verma-Jain       & 33 & \cite[pp. 255-256, (3.49)--(3.53)]{vj1}\\
  $(3,2,5)$ &Verma-Jain       & 51 & \cite[pp. 34-36, (8.18)--(8.29)]{vj2}\\
  $(3,2,6)$ &Verma-Jain       & 57 & \cite[pp. 32-34, (8.6)--(8.17)]{vj2}\\
  $(4,1,6)$ &A.V. Sills        & 52 & \cite[(A.25)]{avs:rrt}\\ 
  $(6,1,8)$ &Verma-Jain       &102 & \cite[pp. 40-41, (8.47)--(8.54)]{vj2}\\
  $(6,1,9)$ &Verma-Jain       &114 & \cite[pp. 38-39, (8.37)--(8.45)]{vj2}\\
 
 \hline
\end{tabular}

 \subsubsection{The five identities associated with $(d,e,k)=(1,2,4)$}\label{fullex}
 Let us now demonstrate in detail how a particular family of Rogers-Ramanujan-Bailey
 type identities is deduced.
Begin by noting that \cite[equation (3.9)]{avs:rrb} states 
 \begin{equation}\label{BP124}
\beta^{(1,2,4)}_n (a^2,0,q^2) =\sum_{r\geqq 0} 
  \frac{a^{2r} q^{2r^2}}{(-aq;q)_{2r} (q^2;q^2)_{r} (q^2;q^2)_{n-r}}.
\end{equation}
 Inserting (\ref{BP124}) into (\ref{WBL}) yields an $a$-generalization of
 a Rogers-Ramanujan type identity related to the modulus 11 due to Dennis Stanton
 ~\cite[p. 65, equation (6.4)]{ds}.  
Stanton presents this as an isolated identity, but
actually it is one of a family of $k+d(e-1)=5$ closely related identities.
Define
  \begin{equation} \label{F124-5}
    F^{(1,2,4)}_5 (a) := \sum_{n,r\geqq 0} \frac{a^{2n+4r} q^{2n^2 + 4nr + 4r^2}}
    {(-aq;q)_{2r} (q^2;q^2)_r (q^2;q^2)_n}.
  \end{equation}
  
  \begin{rmk} In order to obtain (\ref{F124-5}) from the $(d,e,k)=(1,2,4)$ case of
  (\ref{Flastgen}), the order of summation is reversed and $n$ is replaced by $n+r$.
  \end{rmk}
  
 Next, set up the system of $q$-difference equations:
 \begin{eqnarray}
  F^{(1,2,4)}_{1} (a) &=& \frac{1}{(1+aq)} F^{(1,2,4)}_5 (aq)
   \label{F124qde1}\\
  F^{(1,2,4)}_{2} (a) &=& F^{(1,2,4)}_1 (a)
    + \frac{ a q}{(1+aq)} F^{(1,2,4)}_4 (aq) \label{F124qde2}\\
   F^{(1,2,4)}_{3} (a) &=& F^{(1,2,4)}_2 (a)
    + \frac{ a^2 q^2}{(1+aq)} F^{(1,2,4)}_3 (aq) \label{F124qde3} \\
   F^{(1,2,4)}_{4} (a) &=& F^{(1,2,4)}_3 (a)
    + \frac{ a^3 q^3}{(1+aq)} F^{(1,2,4)}_2 (aq) \label{F124qde4} \\ 
   F^{(1,2,4)}_{5} (a) &=& F^{(1,2,4)}_4 (a)
    + \frac{ a^4 q^4}{(1+aq)} F^{(1,2,4)}_1 (aq) \label{F124qde5}.  
\end{eqnarray}
Given that we have an explicit formula for $F^{(1,2,4)}_5 (a)$, it is thus possible
to find, one by one, explicit formulas for each of the other $F^{(1,2,4)}_i (a)$.

Using (\ref{F124qde1}), it is immediate that
\begin{equation} \label{F124-1}
  F^{(1,2,4)}_1 (a) = \sum_{n,r\geqq 0} \frac{a^{2n+4r} q^{2n^2 + 4nr + 4r^2+2n+4r}}
    {(-aq;q)_{2r+1} (q^2;q^2)_r (q^2;q^2)_n}.
  \end{equation}
Next, using (\ref{F124qde5}) we find
\begin{eqnarray}
F^{(1,2,4)}_4 (a) &=& F^{(1,2,4)}_5 (a) - \frac{ a^4 q^4}{(1+aq)} F^{(1,2,4)}_1 (aq)
\nonumber\\
&=& \sum_{n,r\geqq 0} \frac{a^{2n+4r} q^{2n^2 + 4nr + 4r^2}}
    {(-aq;q)_{2r} (q^2;q^2)_r (q^2;q^2)_n}
    -\sum_{n,r\geqq 0} \frac{a^{2n+4r+4} q^{2n^2 + 4nr + 4r^2+4n+8r+4}}
    {(-aq;q)_{2r+2} (q^2;q^2)_r (q^2;q^2)_n} \nonumber\\
&=&\sum_{n,r\geqq 0} \frac{a^{2n+4r} q^{2n^2 + 4nr + 4r^2}}
    {(-aq;q)_{2r} (q^2;q^2)_r (q^2;q^2)_n}
    -\sum_{n,r\geqq 0} \frac{a^{2n+4r} q^{2n^2 + 4nr + 4r^2}}
    {(-aq;q)_{2r} (q^2;q^2)_{r-1} (q^2;q^2)_n} \nonumber\\
&=&\sum_{n,r\geqq 0} \frac{a^{2n+4r} q^{2n^2 + 4nr + 4r^2}}
    {(-aq;q)_{2r} (q^2;q^2)_r (q^2;q^2)_n}\big(1-(1-q^{2r})\big)\nonumber\\
&=&\sum_{n,r\geqq 0} \frac{a^{2n+4r} q^{2n^2 + 4nr + 4r^2+2r}}
    {(-aq;q)_{2r} (q^2;q^2)_r (q^2;q^2)_n},    
\end{eqnarray}
where the third equality follows by simply by shifting the index $r$ to $r-1$ in the
second summation.
 
   Using (\ref{F124qde2}), we find
\begin{equation}\label{F124-2}
 F^{(1,2,4)}_2 (a) = \sum_{n,r\geqq 0} \frac{a^{2n+4r} q^{2n^2 + 4nr + 4r^2+2n+4r}}
    {(-aq;q)_{2r} (q^2;q^2)_r (q^2;q^2)_n}.
 \end{equation}
 
 And finally, via (\ref{F124qde4}),
 \begin{eqnarray}
 F^{(1,2,4)}_3 (a) &=& \sum_{n,r\geqq 0} \frac{a^{2n+4r} q^{2n^2 + 4nr + 4r^2+2r}}
    {(-aq;q)_{2r} (q^2;q^2)_r (q^2;q^2)_n}
    - \sum_{n,r\geqq 0} \frac{a^{2n+4r+3} q^{2n^2 + 4nr + 4r^2+4n+8r+3}}
    {(-aq;q)_{2r+1} (q^2;q^2)_r (q^2;q^2)_n}\nonumber\\
 &=& \sum_{n,r\geqq 0} \frac{a^{2n+4r} q^{2n^2 + 4nr + 4r^2+2r}}
    {(-aq;q)_{2r} (q^2;q^2)_r (q^2;q^2)_n}
    - \sum_{n,r\geqq 0} \frac{a^{2n+4r+1} q^{2n^2 + 4nr + 4r^2+4r+1}}
    {(-aq;q)_{2r+1} (q^2;q^2)_r (q^2;q^2)_n}\nonumber\\   
 &=& \sum_{n,r\geqq 0} \frac{a^{2n+4r} q^{2n^2 + 4nr + 4r^2+2r}(1+aq^{2n+2r+1})}
    {(-aq;q)_{2r+1} (q^2;q^2)_r (q^2;q^2)_n}. \label{F124-3}
 \end{eqnarray}
 It is easily checked that the initial conditions 
 $F^{(1,2,4)}_1 (0) = F^{(1,2,4)}_2 (0) = F^{(1,2,4)}_3 (0) =
 F^{(1,2,4)}_4 (0) = F^{(1,2,4)}_5 (0)$ are satisfied, thus we have 
 \begin{equation} \label{FQ124} F^{(1,2,4)}_i (a) = Q^{(1,2,4)}_i (a) 
 \end{equation} for $1\leqq i\leqq 5$.
 By setting $a=1$ in (\ref{FQ124}), with the aid of Proposition~\ref{Q1},
 we obtain the following family of Rogers-Ramanujan type identities:
 \begin{eqnarray}
 \sum_{n,r\geqq 0} \frac{q^{2n^2 + 4nr + 4r^2+2n+4r}}
    {(-q;q)_{2r+1} (q^2;q^2)_r (q^2;q^2)_n} = 
    \frac{(q,q^{10},q^{11};q^{11})_\infty}{(q^2;q^2)_\infty}\\
    \sum_{n,r\geqq 0} \frac{q^{2n^2 + 4nr + 4r^2+2n+4r}}
    {(-q;q)_{2r} (q^2;q^2)_r (q^2;q^2)_n} 
    = \frac{(q^2,q^9,q^{11};q^{11})_\infty}{(q^2;q^2)_\infty}\\
   \sum_{n,r\geqq 0} \frac{ q^{2n^2 + 4nr + 4r^2+2r}(1+q^{2n+2r+1})}
    {(-q;q)_{2r+1} (q^2;q^2)_r (q^2;q^2)_n} 
     = \frac{(q^3,q^8,q^{11};q^{11})_\infty}{(q^2;q^2)_\infty}\\
   \sum_{n,r\geqq 0} \frac{q^{2n^2 + 4nr + 4r^2+2r}}
    {(-q;q)_{2r} (q^2;q^2)_r (q^2;q^2)_n}  
   = \frac{(q^4,q^7,q^{11};q^{11})_\infty}{(q^2;q^2)_\infty} \\ 
   \sum_{n,r\geqq 0} \frac{ q^{2n^2 + 4nr + 4r^2}}
    {(-q;q)_{2r} (q^2;q^2)_r (q^2;q^2)_n}  
    = \frac{(q^5,q^6,q^{11};q^{11})_\infty}{(q^2;q^2)_\infty}
 \end{eqnarray}
 Note that this family of five identities related to the modulus 11 is different from those of 
 Andrews~\cite[p. 4082, equation (1.7) with $k=5$]{GEA:oddmoduli}, and 
 \cite[pp. 332-333, equations (1.10)--(1.14)]{GEA:mod11}.

\subsubsection{The family of four identities associated with $(d,e,k)=(1,2,3)$}
In a completely analogous manner, additional families of results may be deduced.
\begin{equation}\label{BP123}
\beta^{(1,2,3)}_n (a^2,0,q^2) = \frac{1}{(-q;q)_n} \sum_{r\geqq 0} 
  \frac{a^r q^{r^2}}{(q;q)_r (q;q)_{n-r} (-aq;q)_{n+r}}.
 \end{equation} by \cite[equation (3.8)]{avs:rrb}.
 
  \begin{gather}
\sum_{n,r\geqq 0} \frac{q^{2n^2+3r^2+4nr+2r+3r}}
{(-q;q)_{n+r}(-q;q)_{n+2r+1}
(q;q)_r (q;q)_n} = \frac{(q,q^8,q^9;q^9)_\infty}{(q^2;q^2)_\infty} \\
\sum_{n,r\geqq 0} \frac{q^{2n^2+3r^2+4nr+2r+3r} (1+q^{r+1}-q^{n+r+1}+q^{n+2r+1})}
{(-q;q)_{n+r}(-q;q)_{n+2r+1}
(q;q)_r (q;q)_n} = \frac{(q^2,q^7,q^9;q^9)_\infty}{(q^2;q^2)_\infty}\\
\sum_{n,r\geqq 0} \frac{q^{2n^2+3r^2+4nr+r} (1-q^n+q^{n+r})}
{(-q;q)_{n+r}(-q;q)_{n+2r}
(q;q)_r (q;q)_n} = \frac{(q^3;q^3)_\infty}{(q^2;q^2)_\infty} \\
\sum_{n,r\geqq 0} \frac{q^{2n^2+3r^2+4nr} }
{(-q;q)_{n+r}(-q;q)_{n+2r}
(q;q)_r (q;q)_n} = \frac{(q^4,q^5,q^9;q^9)_\infty}{(q^2;q^2)_\infty} 
   \end{gather}

\subsubsection{The family of three identities associated with $(d,e,k)=(1,3,1)$}

\begin{equation}\label{BP131}
  \beta^{(1,3,1)}_n (a^3,0,q^3) = 
\frac{(-1)^n a^{-n} q^{-\binom{n+1}{2}}(q;q)_n}{(q^3;q^3)_n (aq;q)_{2n}}
\sum_{r\geqq 0} \frac{(-1)^r q^{\binomial{r+1}{2}-nr}(aq;q)_{n+r} (aq;q)_{2n+r}}
{(q;q)_r (q;q)_{n-r} (a^3 q^3;q^3)_{n+r} }
\end{equation} by~\cite[equation (3.10)]{avs:rrb}.

\begin{gather}
\sum_{n,r\geqq 0} \frac{(-1)^n q^{\frac 52 n^2 + 4nr + 2r^2 +\frac 32 n + 2r}
(q;q)_{2n+3r+1} (q;q)_{n+2r+1} (q;q)_{n+r} }
{(q^3;q^3)_{n+r} (q;q)_n (q;q)_r (q^3;q^3)_{n+2r+1} (q;q)_{2n+2r+1} }
= \frac{(q,q^6,q^7;q^7)_\infty}{(q^3;q^3)_\infty}\\
1-\sum_{n,r\geqq 0} \frac{(-1)^n q^{\frac 52 n^2 + 4nr + 2r^2 + \frac 92 n
+ 5r + 2}(q;q)_{2n+3r+1} (q;q)_{n+r+1} (q;q)_{n+2r+1} (1+q^{n+r+1}) }
{(q^3;q^3)_{n+r+1} (q;q)_n (q;q)_r (q^3;q^3)_{n+2r+1} (q;q)_{2n+2r+2} } 
= \frac{(q^2,q^5,q^7;q^7)_\infty}{(q^3;q^3)_\infty}\\
\sum_{n,r\geqq 0} \frac{(-1)^n q^{\frac 52 n^2 + 4nr + 2r^2 - \frac 12 n}
(q;q)_{2n+3r} (q;q)_{n+2r} (q;q)_{n+r} }
{(q;q)_{2n+2r} (q;q)_n (q;q)_r (q^3;q^3)_{n+2r} (q^3;q^3)_{n+r}}
=\frac{(q^3,q^4,q^7;q^7)_\infty}{(q^3;q^3)_\infty}
\end{gather}

\subsubsection{The family of four identities associated with $(d,e,k)=(1,4,1)$}

 \begin{equation}\label{BP141}
 \beta^{(1,4,1)}_n (a^4,0,q^4) = \frac{(-1)^n q^{2n^2}}{(-a^2 q^2;q^2)_{2n}}
  \sum_{r\geqq 0} \frac{ q^{3r^2-4nr} }{(q^2;q^2)_r (-aq;q)_{2r} (q^4;q^4)_{n-r}}
 \end{equation} by~\cite[equation (3.13)]{avs:rrb}.
 \begin{gather}
\sum_{n,r\geqq 0}
\frac{(-1)^{n+r} q^{6n^2+8nr+5r^2+4n+4r}}{(-q^2;q^2)_{2n+2r+1} (q^2;q^2)_r 
(-q;q)_{2r+1} (q^4;q^4)_n}
= \frac{(q,q^8,q^9;q^9)_\infty}{(q^4;q^4)_\infty}\\
\sum_{n,r\geqq 0}
\frac{(-1)^{n+r} q^{6n^2+8nr+5r^2+4n+2r}}{(-q^2;q^2)_{2n+2r+1} 
(q^2;q^2)_r (-q;q)_{2r} (q^4;q^4)_n}
= \frac{(q^2,q^7,q^9;q^9)_\infty}{(q^4;q^4)_\infty}\\
\sum_{n,r\geqq 0}
\frac{(-1)^{n+r} q^{6n^2+8nr+5r^2-2r}}{(-q^2;q^2)_{2n+2r} (q^2;q^2)_r (-q;q)_{2r}
(q^4;q^4)_n}
= \frac{(q^3;q^3)_\infty}{(q^4;q^4)_\infty}\\
\sum_{n,r\geqq 0}
\frac{(-1)^{n+r} q^{6n^2+8nr+5r^2}}{(-q^2;q^2)_{2n+2r} (q^2;q^2)_r (-q;q)_{2r}
(q^4;q^4)_n}
= \frac{(q^4,q^5,q^9;q^9)_\infty}{(q^4;q^4)_\infty}
\end{gather}

\subsubsection{The family of seven identities associated with $(d,e,k)=(1,4,4)$}
\begin{equation}\label{BP144}
 \beta^{(1,4,4)}_n (a^4,0,q^4) = \frac{1}{(-a^2 q^2;q^2)_{2n}}
  \sum_{r\geqq 0} \frac{ a^{2r} q^{2r^2} }{(q^2;q^2)_r (-aq;q)_{2r} (q^4;q^4)_{n-r}}
 \end{equation} by \cite[equation (3.16)]{avs:rrb}. 
 
 \begin{gather}
\sum_{n,r\geqq 0}
 \frac{q^{4n^2 + 8nr + 6r^2+4n+6r}}
 {(-q^2;q^2)_{2n+2r+1} (q^2;q^2)_r (-q;q)_{2r+1} (q^4;q^4)_n}
= \frac{(q,q^{14},q^{15};q^{15})_\infty}{(q^4;q^4)_\infty}\\
\sum_{n,r\geqq 0}
 \frac{q^{4n^2 + 8nr + 6r^2+4n+6r}}
 {(-q^2;q^2)_{2n+2r+1} (q^2;q^2)_r (-q;q)_{2r} (q^4;q^4)_n}
= \frac{(q^2,q^{13},q^{15};q^{15})_\infty}{(q^4;q^4)_\infty}\\
\sum_{n,r\geqq 0}
 \frac{q^{4n^2 + 8nr + 6r^2+4n+6r} (1+q^{2r+1}+2q^{2r+2} +q^{4r+3}+ q^{4n+4r+4})}
 {(-q^2;q^2)_{2n+2r+1} (q^2;q^2)_r (-q;q)_{2r+2} (q^4;q^4)_n}
= \frac{(q^3,q^{12},q^{15};q^{15})_\infty}{(q^4;q^4)_\infty}\\
\sum_{n,r\geqq 0}
 \frac{q^{4n^2 + 8nr + 6r^2-2} 
 (q^{4r}+q^{4r+2} +q^{6r+1} + q^{4n+6r+3}-1-q^{2r+1})}
 {(-q^2;q^2)_{2n+2r} (q^2;q^2)_r (-q;q)_{2r+1} (q^4;q^4)_n}
= \frac{(q^4,q^{11},q^{15};q^{15})_\infty}{(q^4;q^4)_\infty}\\
\sum_{n,r\geqq 0}
 \frac{q^{4n^2 + 8nr + 6r^2+2r} (1+q^{4n+2r+1})}{(-q^2;q^2)_{2n+2r} (q^2;q^2)_r (-q;q)_{2r+1} 
 (q^4;q^4)_n}
= \frac{(q^5;q^5)_\infty}{(q^4;q^4)_\infty}\\
\sum_{n,r\geqq 0}
 \frac{q^{4n^2 + 8nr + 6r^2+2r}}{(-q^2;q^2)_{2n+2r} (q^2;q^2)_r (-q;q)_{2r} (q^4;q^4)_n}
= \frac{(q^6,q^9,q^{15};q^{15})_\infty}{(q^4;q^4)_\infty}\\
 \sum_{n,r\geqq 0}
 \frac{q^{4n^2 + 8nr + 6r^2}}{(-q^2;q^2)_{2n+2r} (q^2;q^2)_r (-q;q)_{2r} (q^4;q^4)_n}
= \frac{(q^7,q^8,q^{15};q^{15})_\infty}{(q^4;q^4)_\infty}
\end{gather}

 \subsubsection{The familiy of four identities associated with $(d,e,k)=(2,2,2)$.}
 \begin{equation}\label{BP222}
\beta^{(2,2,2)}_n (a^2,0,q^2) =\frac{(-1)^n q^{n^2}}{(-aq;q)_{2n} }
\sum_{r\geqq 0} 
  \frac{(-1)^r q^{\frac 32 r^2 - \frac 12 r - 2nr}}
{(aq;q^2)_{r} (q;q)_{r} (q^2;q^2)_{n-r}}
 \end{equation} by~\cite[equation (3.20)]{avs:rrb}.
 
 \begin{gather}
\sum_{n,r\geqq 0} 
  \frac{ (-1)^n q^{3n^2+4n+4nr+\frac 52 r^2+\frac 72 r}}{(-q;q)_{2n+2r+2}  (q;q)_r 
  (q;q^2)_{r+1} (q^2;q^2)_n}
 = \frac{(q^2,q^{16},q^{18};q^{18})_\infty}{(q^2;q^2)_\infty}\\
 \sum_{n,r\geqq 0} 
  \frac{ (-1)^n q^{3n^2+4n+4nr+\frac 52 r^2+4n+\frac 32 r}(q^{2r+2}-q+q^r+q^{r+1})}
  {(-q;q)_{2n+2r+2}  (q;q)_r  (q;q^2)_{r+1} (q^2;q^2)_n}
 = \frac{(q^4,q^{14},q^{18};q^{18})_\infty}{(q^2;q^2)_\infty}\\
 \sum_{n,r\geqq 0} 
  \frac{ (-1)^n q^{3n^2+4nr+\frac 52 r^2-\frac 52 r}(q^{r}+q^{r+1}-q)}
  {(-q;q)_{2n+2r}  (q;q)_r  (q;q^2)_r (q^2;q^2)_n}
 = \frac{(q^6;q^6)_\infty}{(q^2;q^2)_\infty}\\
\sum_{n,r\geqq 0} 
  \frac{ (-1)^n q^{3n^2+4nr+\frac 52 r^2-\frac r2}}{(-q;q)_{2n+2r}  (q;q)_r 
  (q;q^2)_r (q^2;q^2)_n}
 = \frac{(q^8,q^{10},q^{18};q^{18})_\infty}{(q^2;q^2)_\infty}
\end{gather} 
 
\subsubsection{The familiy of five identities associated with $(d,e,k)=(2,2,3)$.}
\begin{equation}\label{BP223}
\beta^{(2,2,3)}_n (a^2,0,q^2) =
\frac{(aq^2;q^2)_n}{(a^2 q^2;q^2)_{2n} }\sum_{r\geqq 0} 
  \frac{a^{r} q^{2nr}}{(q^2;q^2)_{r}  (q^2;q^2)_{n-r}}
 \end{equation} by~\cite[equation (3.21)]{avs:rrb}.

\begin{gather}
\sum_{n,r\geqq 0} \frac{q^{n^2+2r^2+3nr+2n+3r} (q;q)_{n+r+1}}
{(q;q)_{2n+2r+2} (q;q)_r (q;q)_n} 
= \frac{(q,q^{10},q^{11};q^{11})_\infty}{(q;q)_\infty}\\
\sum_{n,r\geqq 0} \frac{q^{n^2+2r^2+3nr+2n+3r} (q;q)_{n+r+1} 
(1+q^{n+1}+q^n-q^{2n})}
{(q;q)_{2n+2r+2} (q;q)_r (q;q)_n} 
= \frac{(q^2,q^9,q^{11};q^{11})_\infty}{(q;q)_\infty}\\
\sum_{n,r\geqq 0} \frac{q^{n^2+2r^2+3nr} (q;q)_{n+r} 
[(q^{n-1}+q^n-q^{2n-1})-(q^{-1}+q^{n-2}+q^{n-1}-q^{2n-4})(1-q^{n-1})(1-q^n)]}
{(q;q)_{2n+2r} (q;q)_r (q;q)_n} \nonumber\\
= \frac{(q^3,q^8,q^{11};q^{11})_\infty}{(q;q)_\infty}\label{messy}\\
\sum_{n,r\geqq 0} \frac{q^{n^2+2r^2+3nr} (q;q)_{n+r} 
(q^{n-1}+q^n-q^{2n-1})}
{(q;q)_{2n+2r} (q;q)_r (q;q)_n} 
= \frac{(q^4,q^7,q^{11};q^{11})_\infty}{(q;q)_\infty}\\
\sum_{n,r\geqq 0} \frac{q^{n^2+2r^2+3nr} (q;q)_{n+r} }
{(q;q)_{2n+2r} (q;q)_r (q;q)_n} 
= \frac{(q^5,q^6,q^{11};q^{11})_\infty}{(q;q)_\infty}
\end{gather}

\subsubsection{The familiy of six identities associated with $(d,e,k)=(2,2,4)$.}
\begin{equation}\label{BP224}
\beta^{(2,2,4)}_n (a^2,0,q^2) =
\frac{(aq^2;q^2)_n}{(a^2 q^2;q^2)_{2n} }\sum_{r\geqq 0} 
  \frac{a^{r} q^{2r^2}}{(q^2;q^2)_{r}  (q^2;q^2)_{n-r}}
 \end{equation} by~\cite[equation (3.22)]{avs:rrb}.

\begin{gather}
\sum_{n,r\geqq 0} \frac{q^{n^2+2r^2+2nr+2n+3r} (q;q)_{n+r+1}}
{(q;q)_{2n+2r+2} (q;q)_r (q;q)_n} 
= \frac{(q,q^{12},q^{13};q^{13})_\infty}{(q;q)_\infty}\\
\sum_{n,r\geqq 0} \frac{q^{n^2+2r^2+2nr+2n+3r} (q;q)_{n+r+1}
(1+q^{r+1}+q^{n+1}-q^{n+r+1})}
{(q;q)_{2n+2r+2} (q;q)_r (q;q)_n} 
= \frac{(q^2,q^{11},q^{13};q^{13})_\infty}{(q;q)_\infty}\\
\sum_{n,r\geqq 0} \frac{q^{n^2+2r^2+2nr+2n+3r} (q;q)_{n+r+1}
}{(q;q)_{2n+2r+2} (q;q)_r (q;q)_n} 
\Big( (1-q^{n+r+1})(q^r+q^n)(1+q) - q^{n+r} + q^{2r+2} + q^{2n+2} + q^{2n+2r+1}
+ q^{n+r+2} \Big)\nonumber\\
= \frac{(q^3,q^{10},q^{13};q^{13})_\infty}{(q;q)_\infty} \label{ugly2}\\
\sum_{n,r\geqq 0} \frac{q^{n^2+2r^2+2nr-1} (q;q)_{n+r}
}{(q;q)_{2n+2r} (q;q)_r (q;q)_n} 
\Big( q^r-1+q^{2r+1}+q^{n+r+1}-q^{n+2r}+q^n+q^{2n+1}-q^{2n+r}-q^{n+2r+1}
-q^{2n+r+1}+q^{2n+2r} \Big)\nonumber\\
= \frac{(q^4,q^9,q^{13};q^{13})_\infty}{(q;q)_\infty}\label{ugly3}\\
\sum_{n,r\geqq 0} \frac{q^{n^2+2r^2+2nr} (q;q)_{n+r}
(q^n+q^r-q^{n+r})}{(q;q)_{2n+2r} (q;q)_r (q;q)_n} 
= \frac{(q^5,q^8,q^{13};q^{13})_\infty}{(q;q)_\infty}\\
\sum_{n,r\geqq 0} \frac{q^{n^2+2r^2+2nr} (q;q)_{n+r}
}{(q;q)_{2n+2r} (q;q)_r (q;q)_n} 
= \frac{(q^6,q^7,q^{13};q^{13})_\infty}{(q;q)_\infty}
\end{gather}
 
\section{Conclusion}\label{conc}
A full family of $k+d(e-1)$ identities could be worked out for any $(d,e,k)$.  
Of course, the serendipitous cancellations which arise in the calculations of
\S\ref{fullex} are not guaranteed to occur everywhere.  For instance, despite my best
efforts, I could not find neater looking representations for the left hand sides
of (\ref{messy}), (\ref{ugly2}), and (\ref{ugly3}). 
 Nonetheless, in many of the cases explored here, I was fortunate
enough to be able to add some attractive identities to the literature.

Furthermore, it should be noted that the 
$Q^{(d,e,k)}_i(a)$ studied here generalize the
$Q_{d,k,i} (a)$ I studied 
previously in \cite{avs:rrt}, which in turn generalizes some of the
ideas in Andrews~\cite{gea:qDiff}.   In fact, the $Q_{d,k,i}(a)$ of
\cite{avs:rrt} is precisely the $e=1$ case of the $Q^{(d,e,k)}_i(a)$ in
this present work.

  The $q$-difference equations associated with the $e=1$ case led to a breakthough in
understanding of the combinatorics of a large class of both new and classical
Rogers-Ramanujan type identities~\cite{avs:rrt}. 
It is hoped that the $q$-difference equations herein can help shed some light
on the combinatorics of the identities for additional values of $e$, or even
more optimistically, for general $e$.

\end{document}